\documentclass[a4paper]{article}
\usepackage{amsfonts}
\usepackage{color}
\usepackage{epsfig}
\usepackage{amsmath,amssymb,amsthm}
\usepackage[T1]{fontenc}
\usepackage[utf8]{inputenc}
\usepackage{xcolor}
\usepackage{geometry}
\newtheorem{theo}{Theorem}[section]
\newtheorem{lem}[theo]{Lemma}
\newtheorem{cor}[theo]{Corollary}

\newcommand{\omeps}{\omega_{\eps}}
\newcommand{\ue}{u_{\eps}}
\newcommand{\ve}{v_{\eps}}

\newcommand{\ye}{y_{\eps}}

\newcommand{\Lom}[1]{L^{#1}(\Omega)}
\newcommand{\norm}[2][]{\|#2\|_{#1}}

\newcommand{\Om}{\Omega}
\newcommand{\Ombar}{\overline{\Omega}}
\newcommand{\na}{\nabla}
\newcommand{\delny}{\partial_{\nu}}
\newcommand{\bdry}{\big\rvert_{\partial\Om}}
\newcommand{\muhat}{\widehat{\mu}}

\numberwithin{equation}{section}

\newcommand{\be}{\begin{equation} \label}
\newcommand{\ee}{\end{equation}}
\newcommand{\bea}{\begin{eqnarray}\label}
\newcommand{\eea}{\end{eqnarray}}
\newcommand{\bas}{\begin{eqnarray*}}
\newcommand{\eas}{\end{eqnarray*}}
\newcommand{\bit}{\begin{itemize}}
\newcommand{\eit}{\end{itemize}}
\newcommand{\nn}{\nonumber}
\newcommand{\lp}[2]{\Vert{#2}\Vert_{L^{#1}(\Omega)}}
\newcommand{\R}{\mathbb{R}}
\newcommand{\N}{\mathbb{N}}
\newcommand{\pO}{\partial\Omega}

\newcommand{\eps}{\varepsilon}

\newcommand{\io}{\int_\Omega}

\newcommand{\abs}{\\[5pt]}

\newcommand{\tmax}{T_{\rm max,\eps}}

\newcommand{\uinit}{u_{init}}
\newcommand{\vinit}{v_{init}}
\begin{document}
\title{How far does small chemotactic interaction 
perturb \\ the Fisher--KPP dynamics?}
\author{
Johannes Lankeit\footnote{jlankeit@math.uni-paderborn.de}\\
{\small Institut f\"ur Mathematik, Universit\"at Paderborn,}\\
{\small Warburger Str. 100, 33098 Paderborn, Germany}
\and
Masaaki Mizukami\footnote{masaaki.mizukami.math@gmail.com}\\
{\small Department of Mathematics, Tokyo University of Science,}\\
{\small 1-3, Kagurazaka, Shinjuku-ku, Tokyo 162-8601, Japan}
}
\date{}
\maketitle
\begin{abstract}
\noindent 
  This paper deals with 
  nonnegative solutions 
  of the Neumann initial-boundary value problem 
  for the fully parabolic chemotaxis-growth system,
  \bas
	\left\{ \begin{array}{lr}
	(u_{\eps})_t=\Delta u_{\eps} - \eps \nabla \cdot ( u_\eps \nabla v_\eps) + \mu u_\eps(1 - u_\eps),
	\qquad& x\in \Omega, \ t>0, \\[1mm]
	(v_{\eps})_t=\Delta v_\eps -v_\eps+u_\eps,
	\qquad& x\in \Omega, \ t>0, 
	\end{array} \right.
  \eas
  with positive small parameter $\eps>0$ 
  in a bounded convex domain 
  $\Omega\subset\R^n$ ($n\geq 1$)
  with smooth boundary. 
  The solutions converge to the solution $u$ to the 
  Fisher--KPP equation as $\eps\to 0$. 
  It is shown that 
  for all $\mu>0$ and any suitably regular 
  nonnegative initial data
  $(\uinit,\vinit)$ 
  there are some constants 
  $\eps_0>0$  and 
  $C>0$ such that
  \bas
    \sup_{t>0}
    \lp{\infty}{u_\eps(\cdot,t)-u(\cdot,t)}
    \leq C\eps
  \quad for\ all\ \eps\in(0,\eps_0).
  \eas
  \abs
%
%
\end{abstract}
%
%
%
%
%
%
%
%
\newpage
\section{Introduction}\label{intro}
The Fisher-KPP equation (\cite{fisher,KPP}) 
\begin{equation}\label{fkpp}
 u_t=\Delta u + \mu u(1-u) 
\end{equation}
modelling the spread and growth of a biological population - or, in the original setting, of the prevalence of an advantageous gene within the population (\cite{fisher}) - is well-studied and clearly of interest on its own, and there is a large corpus of literature bearing witness to this, ranging from articles on existence and speed (\cite{fisher,KPP}) or stability (\cite{huang}) of travelling waves, long-term behaviour of solutions and a 'hair-trigger effect' (i.e. instability of the rest state $u\equiv0$) (\cite{aronson_weinberger}) 
to 
treatments of system variants 
in a heterogeneous environment (\cite{berestycki_hamel_nadirashvili}), 
with nonlinear (\cite{lau}) or fractional (\cite{stan_vazquez}) diffusion or nonlocal interaction (\cite{hamel_ryzhik}),  
spatio-temporal delays (\cite{ai}), or investigations of the spreading as free boundary problem (\cite{du_guo}), to name but a few.

At the same time 
owing to the rather simple structure of the equation, 
it is no wonder that \eqref{fkpp} 
makes an appearance as constituent 
of more complex models. 

In the present article, we shall view chemotaxis systems with logistic growth terms as perturbation of \eqref{fkpp} and ask to what extent the behaviour of solutions to \eqref{fkpp} is altered in the presence of weak chemotactic effects.\\[5pt]

\textbf{Chemotaxis models with growth terms.} 
The Keller-Segel model (\cite{KS}, see also the surveys \cite{horstmannI,BBTW}) has arisen from the ambition to understand chemotaxis, i.e. the partially directed movement of cells (bacteria, slime mould, etc.; with density denoted by $u$) in the direction indicated by concentration gradient of a signal substance (concentration $v$) they themselves produce:

\begin{equation}\label{KS}
	\left\{ \begin{array}{ll}
	u_t=\Delta u - \chi
         \nabla \cdot (u \nabla v) 
        + f(u),
\\[1mm]
	\tau v_t=\Delta v -v+u,
	\end{array} \right.
\end{equation}
Herein, the constant $\chi$ stands for the chemotactic sensitivity and 
$f$ is used to denote growth terms, one of the most natural forms (apart from $f\equiv 0$) being $f(u)=u-u^2$. 

The model plays an important role in the mathematical study of emergence of pattern and structure in many different biological contexts (see \cite{hillenpainter_survey}), e.g. slime mould formation, bacterial patterning, embryonic development, progression of cancer, and has spawned an abundance of mathematical literature over the past decades (see \cite{horstmannI,BBTW} and the references therein). 

In particular in the presence of logistic source terms like $f(u)=\kappa u - \mu u^2$ 
(cf. also \cite[Sec. 2.8]{hillenpainter_survey}) structure formation can be observed in \eqref{KS}, as witnessed by 
the 
numerical experiments in \cite{painterhillen_physD} or \cite{ei_izuhara_mimura}, 
attractor results in \cite{OTYM} and transient growth phenomena demonstrated in \cite{win_exceed,lankeit_thresholds,win_emergence}. 
Let us recall some known results about this system: 

In the simplified parabolic-elliptic case (i.e. $\tau=0$), with $f$ generalizing $f(u)=\kappa u-\mu u^2$,  
Tello and Winkler proved the existence of global weak solutions and global classical solutions if $\mu>\frac{n-2}n$. They also showed convergence of solutions to the constant steady state under a stricter condition on the source terms. Very weak solutions have been constructed for sources of the form $f(u)=\kappa u- \mu u^\alpha$ for $\alpha>2-\frac1n$ in \cite{win_veryweak}.


As to the fully parabolic case of \eqref{KS} ($\tau=1$, again with $f(u)=\kappa u-\mu u^2$) it is known that globally bounded classical solutions exist in two-dimensional domains (\cite{OTYM}) or if $\mu$ is sufficiently large (\cite{W2010}); these solutions converge, provided a further largeness requirement is satisfied by $\mu$ (\cite{W2014}).  
For any positive $\mu$, global weak solutions exist (\cite{lankeit_evsmooth}), which moreover in three-dimensional settings are known to become classical after some waiting time and enter an absorbing ball in $C^{2+\alpha}$ if $\kappa$ is sufficiently small (\cite{lankeit_evsmooth}). 

Extensive studies regarding the interplay of exponents $\alpha$ and $\beta$ with respect to global existence of bounded solutions to  
the system obtained from \eqref{KS} upon replacing the production term $+u$ in the second equation by, roughly speaking, 
$+u^{\beta}$ and with $f(u)=u-u^\alpha$ have been conducted by Nakaguchi and Osaki (\cite{nakaguchi_osaki_2d,nakaguchi_osaki_2d3d}). The existence of very weak solutions to \eqref{KS} 
with $f(u)=u-u^\alpha$, $\alpha>2-\frac1n$, has been established by Viglialoro (\cite{Viglialoro}) for bounded domains of arbitrary dimension. 
The convergence rate of solutions for both the parabolic-elliptic and the parabolic-parabolic variant of \eqref{KS} has recently been studied by He and Zheng (\cite{he_zheng}). \\[5pt]


\textbf{Limiting cases as parameters tend to zero.} Letting different parameters in \eqref{KS} tend to zero can help uncover dynamical properties in \eqref{KS} and the relation to affiliated models. 
In \cite{win_exceed,lankeit_thresholds}, considering $\eps\to 0$ in 

\begin{equation} \label{system:thresholds}
 \left\{\begin{array}{l}
 u_t=\eps\Delta u -\nabla \cdot(u\nabla v) +\kappa u -\mu u^2\\
 0=\Delta v-v+u\nn
 \end{array}\right. 
\end{equation}
was used to obtain insight into some transient growth phenomenon of solutions from a blow-up result in the 
hyperbolic-elliptic limit system \eqref{system:thresholds} with $\eps=0$, both in the one-dimensional (\cite{win_exceed}) and in the higher-dimensional radially symmetric case (\cite{lankeit_thresholds}). 

For quite general choices of $f$, system \eqref{KS} had been suggested and investigated by Mimura and Tsujikawa (\cite{mimura_tsujikawa}). 
Inter alia, they considered the limit $\eps\to 0$ of the time-rescaled system with Allee effect
\begin{equation*}
 \left\{\begin{array}{l}
 u_t=\eps^2 \Delta u - \eps \na \cdot(u\na v) + u(1-u)(u-a),\\
 v_t=\Delta v + u -v,
 \end{array}\right. 
\end{equation*}
$a\in (0,\frac12)$, thus showing the existence of localized aggregating patterns.\\[5pt]


In the present paper, we want to investigate the disturbances to Fisher-KPP dynamics caused by weak chemotactic effects 
 and hence consider the system 

\be{cp}
	\left\{ \begin{array}{lr}
	(u_{\eps})_t=\Delta u_{\eps} - \eps \nabla \cdot ( u_\eps \nabla v_\eps) + \mu u_\eps(1 - u_\eps),
	\qquad& x\in \Omega, \ t>0, \\[1mm]
	(v_{\eps})_t=\Delta v_\eps -v_\eps+u_\eps,
	\qquad& x\in \Omega, \ t>0, \\[1mm]
	\frac{\partial u_\eps}{\partial\nu}=
        \frac{\partial v_\eps}{\partial\nu}=0,
	\qquad& x\in \pO, \ t>0, \\[1mm]
	u_\eps(x,0)=\uinit(x),\;
        v_\eps(x,0)=\vinit(x),    
	\qquad& x\in\Omega, 
	\end{array} \right.
\ee
in a bounded convex domain $\Omega\subset\R^n$ 
($n\in\N$) 
with smooth boundary, 
where  $\mu>0$ and 
\be{init}
\uinit\in C^0(\Ombar) 
\ \mbox{and} \
\vinit\in W^{1,\infty}(\Omega)
\mbox{ are nonnegative} 
\ee 
and where $\eps>0$ is to be small. 
We will compare its solutions to those of 

\be{cp2}
	\left\{ \begin{array}{lr}
	u_{t}=\Delta u + \mu u(1 - u),
	\qquad& x\in \Omega, \ t>0, \\[1mm]
	\frac{\partial u}{\partial\nu}=0,
	\qquad& x\in \pO, \ t>0, \\[1mm]
	u(x,0)=\uinit(x).
	\qquad& x\in\Omega 
	\end{array} \right.
\ee


\textbf{Main result.} Whereas large chemotaxis terms can cause significantly altered solution behaviour (cf. e.g. \cite[Thm. 4.3]{tello_winkler_07}), 
intuition leads to surmise that in presence of weak chemotactic effects, solutions to \eqref{cp} should be close to solutions of \eqref{cp2}. For example, as $\eps\to 0$, one might expect convergence in some $L^p(\Om)$ on each finite time interval. 
We will prove that the solutions converge \textit{uniformly} in $\Om\times(0,\infty)$ and moreover show that this convergence is \textit{linear} in the chemotactic strength $\eps$: 

\begin{theo}\label{maintheorem}
Let $n\in\N$ and 
let $\Omega\subset\R^n$ be a bounded convex domain 
with smooth boundary. 
Let $\mu>0$ and suppose that 
$(\uinit,\vinit)$ satisfies 
(\ref{init}). 
Then there are $\eps_0>0$ and $C>0$ such that 
for all $\eps\in [0,\eps_0)$, (\ref{cp}) has a 
global classical solution 
and that for all $\eps\in(0,\eps_0)$ 
the solution $u_\eps$ to (\ref{cp}) and 
the solution $u$ to (\ref{cp2}) satisfy 
\be{mainineq}
  \sup_{t>0}\|u_\eps(\cdot,t)-u(\cdot,t)\|_{L^{\infty}(\Omega)}
  \leq C\eps. 
\ee 
\end{theo}

\textbf{Strategy and plan of the paper.} 
Having ensured global existence and some uniform bounds for $\ue$ and $\na \ve$ in Section \ref{sec:ge-and-bddness}, in Section \ref{sec:locintimeconv} we will first take care of convergence of $\ue$ to $u$ on finite time intervals $[0,T]$. The key to this proof lies in the derivation of a differential inequality for $\io\omega_\eps^{2k}$ for the difference $\omega_\eps:=\ue-u$ and sufficiently large $k$ (Lemma \ref{lem:diffineq}). This inequality yields an $L^{2k}(\Om)$-estimate 
for $\omega_\eps$, which in Lemma \ref{lem:lqtolinfty} will be turned into a corresponding $L^\infty(\Om)$-information. 
In preparation of a comparison from below (Lemma \ref{lemlowu}), we provide a uniform bound for $\norm[\Lom\infty]{\Delta \ve}$ (Lemma \ref{lemdeltav}) and a positive lower bound for $\ue$ at some positive time (Lemma \ref{lowforeps}). 

The lower estimate obtained from Lemma \ref{lemlowu} can then be inserted into the differential inequality from Lemma \ref{lem:diffineq}, dealing with the difference on the remaining interval $(T,\infty)$ and in Section \ref{sec:proofofthm1} finally proving Theorem \ref{maintheorem}.

%
%
%
\section{Global existence 
and uniform-in-$\eps$ boundedness}\label{sec:ge-and-bddness}
In this section we shall show global existence 
and uniform-in-$\eps$ boundedness 
of solutions to (\ref{cp}). 
Firstly we will recall the well-known result about 
local existence of solutions to (\ref{cp}) 
(see \cite[Lemma 1.1]{W2010}, \cite[Lemma 2.1]{W2014}). 
\begin{lem}\label{local}
Let $\eps\in [0,\infty)$, $\mu>0$, and suppose that 
$(\uinit,\vinit)$ satisfies (\ref{init}). 
Then there exist $\tmax\in(0,\infty]$ and 
a classical solution $(u_\eps,v_\eps)$ of (\ref{cp}) 
in $\Omega\times (0,\tmax)$, which 
satisfy 
\be{crit}
either\ \tmax=\infty \quad or 
\quad \limsup_{t\nearrow \tmax}
\lp{\infty}{u_\eps(\cdot,t)}=\infty.
\ee
Moreover, this solution is uniquely 
determined in the class of function couples such that
  \bas
	& & u_\eps \in C^0(\Ombar\times [0,\tmax)) 
        \cap C^{2,1}(\Ombar\times (0,\tmax))
         \qquad \mbox{and} \nn\\
	& & v_\eps \in C^{0}(\Ombar\times[0,\tmax))
        \cap C^{2,1}(\Ombar\times (0,\tmax))
        \cap L^{\infty}((0,\tmax);W^{1,\infty}(\Omega)).
  \eas
\end{lem}
Throughout the sequel, 
we keep $n\in\N$, 
$\Omega\subset\R^n$, $\mu>0$ and initial data 
$\uinit$ and $\vinit$ satisfying (\ref{init}) fixed and, without loss of generality, assume $\uinit\not\equiv 0$. (If $\uinit\equiv 0$, also $u\equiv 0$ and $\ue\equiv 0$ for any $\eps>0$, and \eqref{mainineq} trivially holds true.) Moreover, 
 we let $\tmax$ 
and $(u_\eps,v_\eps)$ be as given
by Lemma \ref{local}. We also denote the solution 
of (\ref{cp2}) by $u= u_0$. \\[5pt]
To simplify notation we shall abbreviate the deviations from the 
nonzero homogeneous steady state by introducing 
\be{UV}
U_\eps(x,t):=u_\eps(x,t)-1 \quad \mbox{and} \quad 
V_\eps(x,t):=v_\eps(x,t)-1
\ee
for $x\in \overline{\Omega}$ and $t>0$. 
Then by straightforward computation it follows that 
$(U_\eps,V_\eps)$ solves 
\be{cp3}
	\left\{ \begin{array}{lr}
	(U_\eps)_{t}=\Delta U_\eps - \eps \nabla \cdot ( u_\eps\nabla V_\eps) 
        - \mu U_\eps - \mu U_\eps^2,&
	\qquad x\in \Omega, \ t>0, \\[1mm]
	(V_\eps)_{t}=\Delta V_\eps -V_\eps+U_\eps,&
	\qquad x\in \Omega, \ t>0, \\[1mm]
	\frac{\partial u_\eps}{\partial\nu}=
        \frac{\partial v_\eps}{\partial\nu}=0,&
	\qquad x\in \pO, \ t>0, \\[1mm]
	U_\eps(x,0)=\uinit(x)-1,
        V_\eps(x,0)=\vinit(x)-1,    &
	\qquad x\in\Omega. 
	\end{array} \right.
\ee
We will prove global existence and boundedness 
of solutions to (\ref{cp}). 
For the pointwise comparison argument 
(cf. \cite[Lemma 3.1]{W2014}) used in this proof, 
convexity of the domain is essential. 

\begin{lem}\label{lemu}
For any $\eps\in[0,\frac{4\mu}{n})$, 
the solution of (\ref{cp}) exists 
globally. Moreover, there is $c_1>0$ such that 
\be{uineq}
  \lp{\infty}{u_\eps(\cdot,t)}\leq 
  c_1e^{-t}+1
  +\frac{(\mu-1)^2+n\eps}{4\mu-n\eps}
\ee
for all $t>0$ and for all $\eps\in[0,\frac{4\mu}{n})$.
\end{lem}
\proof
With $U_\eps$ and $V_\eps$ as defined in (\ref{UV}), 
we let 
\[z_\eps(x,t):=U_\eps(x,t)+\frac{\eps}{2}
|\nabla V_\eps(x,t)|^2\quad 
\mbox{ for }x\in \overline{\Omega}\mbox{ and }t\in (0,\tmax).\] 
Then $z_\eps$ satisfies 
\bas
  (z_\eps)_t-\Delta z_\eps+z_\eps
  &=&-\eps|D^2V_\eps|^2-\eps u_\eps\Delta V_\eps
  -(\mu-1)U_\eps
  -\mu U_\eps^2
  -\frac{\eps}{2}|\nabla V_\eps|^2
\\
  &\leq& \frac{n\eps}{4}U_\eps^2
  +\frac{n\eps}{2}U_\eps
  +\frac{n\eps}{4}-(\mu-1)U_\eps-\mu U_\eps^2
\\
  &\leq&
  \frac{\left(\frac{n\eps}{2}
  -\mu+1\right)^2}{4\left(\mu-\frac{n\eps}{4}\right)}
  +\frac{n\eps}{4}
\\
  &=&
  \frac{(\mu-1)^2+n\eps}{4\mu-n\eps}
\eas
for all $x\in\Omega$ and $t\in(0,\tmax)$, 
where we have used the condition $4\mu> n\eps$ 
(for more detail, see \cite[Lemma 3.1]{W2014}). 
In order to derive an estimate for $z_\eps$ 
itself from this, we note that since $\Omega$ is convex 
and $\frac{\partial v_\eps}{\partial \nu} = 0$ 
on $\partial \Omega$, we have 
$\frac{\partial |\nabla v_\eps|^2}
{\partial \nu}\leq 0$ 
on $\partial \Omega$ (see \cite[Lemme 2.I.1]{lions}) and 
hence also 
$\frac{\partial z_\eps}{\partial \nu}\leq 0$ 
on $\partial \Omega$. 
We define 
$y_{\eps0}:=\lp{\infty}{U_\eps(\cdot,0)}
+\frac{\eps}{2}\lp{\infty}{\nabla V_\eps(\cdot,0)}^2 > 0$, 
and denote by $y_\eps : [0,\infty) \to \R$ 
the function solving
\bas
\left\{
\begin{array}{l}
  y_\eps'(t)+y_\eps(t)=\frac{(\mu-1)^2+n\eps}{4\mu-n\eps},
  \quad t>0,\\
  y_\eps(0)=y_{\eps 0}. 
\end{array}
\right.
\eas
By the comparison theorem we obtain that 
\bas
  0\leq u_\eps(\cdot,t)=U_\eps(\cdot,t)+1\leq y(t)+1
  \leq 
  c_1e^{-t}+1+\frac{(\mu-1)^2+n\eps}{4\mu-n\eps}, 
\eas
where $c_1:=\lp{\infty}{U_\eps(\cdot,0)}
+\frac{2\mu}{n}\lp{\infty}{\nabla V_\eps(\cdot,0)}^2$. 
In view of (\ref{crit}) we complete the proof, seeing that 
actually $\tmax=\infty$.
\qed

In the next lemma we aim at deriving a bound 
for $\nabla v_\eps$. 
We restrict the admissible values for $\eps$ to a smaller range than in Lemma \ref{lemu} in order to establish the estimate independently of $\eps$, in contrast to the right-hand side of \eqref{uineq}.
%
\begin{lem}\label{lemnablav}
There exist $c_2>0$ and $\lambda_1>0$ such that 
\be{nablavineq}
  \lp{\infty}{\nabla v_\eps(\cdot,t)}\leq 
  c_2 
  (1+e^{-t}+e^{-(1+\lambda_1)t}) 
\ee
for all $t\geq 0$ and all $\eps\in [0,\frac{2\mu}{n})$. 
\end{lem}
\begin{proof}
We note that 
\begin{equation*}
  \frac{(\mu-1)^2+n\eps}{4\mu-n\eps}\leq 
  \frac{(\mu-1)^2+2\mu}{4\mu-2\mu}
  =\frac{\mu^2+1}{2\mu} \qquad \text{ for all } \eps\in\left[0,\frac{2\mu}n\right), 
\end{equation*}
so that according to Lemma \ref{lemu} the estimate 
\bas
  \lp{\infty}{u_\eps(\cdot,t)}\leq 
  c_1 e^{-t}+1+\frac{\mu^2+1}{2\mu}\leq c_1+1+\frac{\mu^2+1}{2\mu}=:c_3
\eas
holds for all $t>0$ and 
all $\eps\in [0,\frac{2\mu}{n})$. 
By means of a variation-of-constants representation 
for $v$, we have
\bas
  \nabla v_\eps(\cdot,t)=\nabla e^{t(\Delta-1)}\vinit+
  \int^t_0\nabla e^{(t-s)(\Delta-1)}u_\eps(\cdot,s)\, ds \qquad \text{for all } t>0.
\eas
Known smoothing estimates for the Neumann heat semigroup in $\Omega$ (more precisely: the limit case $p\to \infty$ in \cite[Lemma 1.3 (iii)]{win_aggregationvs}) provide us with constants $c_4>0$ and $\lambda_1>0$ such that 
\[
 \norm[\Lom\infty]{\na e^{\tau\Delta}\varphi}\leq c_4e^{-\lambda_1\tau}\norm[\Lom\infty]{\na\varphi}\qquad \text{for all } \tau>0 \text{ and all } \varphi\in W^{1,\infty}(\Om).
\]
Accordingly, 
\[
 \norm[\Lom\infty]{\na e^{t(\Delta-1)}\vinit}\leq c_4e^{-(1+\lambda_1)t}\norm[\Lom\infty]{\na\vinit}.
\]

Similarly employing \cite[Lemma 1.3 (ii)]{win_aggregationvs}, 
we gain $c_5>0$ such that
\bas
  \int_0^t\lp{\infty}{\nabla 
  e^{(t-s)(\Delta-1)}\ue(\cdot,s)}\,ds
  &\leq& 
  c_5 \displaystyle\int_0^t 
  (t-s)^{-\frac{1}{2}}e^{-(1+\lambda_1)(t-s)}
  \lp{\infty}{\ue(\cdot,s)}\,ds
\\
  &\leq& 
c_3c_5
\displaystyle\int_0^t 
  (t-s)^{-\frac{1}{2}}e^{-(1+\lambda_1)(t-s)}\,ds
\\
  &\leq& \displaystyle 
  c_3c_5
  \int_0^\infty 
  \sigma^{-\frac{1}{2}}e^{-(1+\lambda_1)\sigma}d\sigma
\eas
for all $t>0$. 
Therefore we have \eqref{nablavineq} 
for all $t>0$ and all $\eps\in [0,\frac{2\mu}{n})$, with an obvious definition of $c_2>0$. 
\end{proof}
%
%
%
%
%
%
\section{Local-in-time convergence to the Fisher--KPP equation}\label{sec:locintimeconv}
In this section we shall prove the convergence of solutions of \eqref{cp}
to those of the Fisher-KPP equation \eqref{cp2} on some interval $[0,T]$. 
We will begin with the key ingredient of both the proof on finite and on eventual time intervals: a differential inequality that will first lead to an estimate of $L^{2k}$-norms of the difference $\omega_\eps$:
%
\begin{lem}\label{lem:diffineq}
 Let $k\geq1$ be an integer. 
 Then there is $c_6(k)>0$ such that for all $\eps\in[0,\frac{2\mu}n)$ 
 \begin{equation*}
  \omeps:= \ue-u\qquad \text{in } \Om\times(0,\infty)
 \end{equation*}
 satisfies 
 \begin{equation}\label{eq:diffineq}
  \frac{d}{dt}\io \omeps^{2k} \leq \eps^{2k} c_6(k) + \mu k\io \omeps^{2k} + 2k\mu \io \omeps^{2k} (1-\ue-u)
 \end{equation}
 on $(0,\infty)$.
\end{lem}
\begin{proof}
We immediately see that 
$\omega_\eps$ satisfies 
\bas
  (\omega_\eps)_t=\Delta \omega_\eps
  -\eps \nabla\cdot(u_\eps\nabla v_\eps)
  +\mu \omega_\eps-\mu(u_\eps+u)\omega_\eps\quad \text{in } \Om\times(0,\infty).
\eas
Multiplying the above equation 
by $\omega_\eps^{2k-1}$ and 
integrating over $\Omega$,
we can calculate 
\begin{align}\label{eq:firstdiffineq}
  \frac{d}{dt}\int_\Omega \omega_\eps^{2k}
  &=-2k(2k-1)\int_\Omega \omega_\eps^{2k-2}|\nabla \omega_\eps|^2
  +\eps 2k(2k-1)\int_\Omega 
  \omega_\eps^{2k-2}u_\eps\nabla \omega_\eps \cdot \nabla v_\eps\nn
\\
  &\qquad\qquad+ 2k\mu \int_\Omega (1-\ue-u)\omega_\eps^{2k} 
\end{align}
on $(0,\infty)$.
Two successive applications of Young's inequality reveal that with some $c_7=c_7(k)>0$ we have 
\begin{align}\label{eq:young1}
 \eps 2k(2k-1)\io \omega_\eps^{2k-2}\ue\na\omega_\eps\cdot\na\ve&\leq 2k(2k-1)\io \omega_\eps^{2k-2}|\na \omega_\eps|^2 + \frac{\eps^2 k(2k-1)}{2} \io \omega_\eps^{2k-2}\ue^2|\na\ve|^2\nn\\
 &\leq 2k(2k-1)\io \omega_\eps^{2k-2}|\na \omega_\eps|^2 +k\mu \io\omega_\eps^{2k} 
 \nn\\
 &\quad\ +c_7(k)\eps^{2k}\io\ue^{2k}|\na\ve|^{2k}
\end{align}
for all $\eps\in[0,\frac{2\mu}{n})$ and on $(0,\infty)$, where thanks to Lemma \ref{lemu} and Lemma \ref{lemnablav} we may further estimate 
\begin{equation}\label{eq:ck} 
 c_7(k)\io \ue^{2k}|\na\ve|^{2k}\leq c_6(k) \qquad \text{ on } (0,\infty)
\end{equation}
for some $c_6(k)>0$ independent of $\eps\in[0,\frac{2\mu}n)$, so that the combination of \eqref{eq:firstdiffineq}, \eqref{eq:young1} and \eqref{eq:ck} finally yields \eqref{eq:diffineq}.
\end{proof}

\begin{cor}\label{cor:norm2k}
 Let $k\geq1$ be an integer and let $c_6(k)$ be as in Lemma \ref{lem:diffineq}. Then, for any $\eps\in[0,\frac{2\mu}n)$, the function $\ue$ satisfies 
 \[
  \norm[\Lom {2k}]{\ue(\cdot,t)-u(\cdot,t)}\leq \sqrt[2k]{c_6(k)} \eps e^{\frac{3\mu}2 t} \qquad \text{for any } t>0.
 \]
\end{cor}
\begin{proof}
 Nonnegativity of $\ue$ and $u$ together with Lemma \ref{lem:diffineq} show that 
 \[
   \frac{d}{dt} \io \omega_\eps^{2k}\leq 3k\mu \io \omega_\eps^{2k} + c_6(k)\eps^{2k}\qquad \text{ on } (0,\infty),
 \]
 which upon an ODE comparison argument and radication readily results in the Corollary, due to the fact that $\omega_\eps(\cdot,0)\equiv 0$.
\end{proof}

We employ semigroup techniques to upgrade these estimates to uniform bounds.

\begin{lem}\label{lem:lqtolinfty}
 Let $q>\frac {n}{2}$ and $p>n$. There is $c_8>0$ such that for any $T>0$ and any $z_0\in C^0(\Ombar)$, $f\in C^0((0,T);C^1(\Ombar;\R^N))$, $g\in C^0(\Ombar\times(0,T))$, the solution $z$ of 
 \[
  z(\cdot,0)=z_0 \text{ in } \Om,\quad \delny z\bdry=0,\quad z_t=\Delta z + \na \cdot f + g \text{ in } \Om\times(0,T)
\]
 for all $t\in(0,T)$ satisfies 
 \[
  \norm[\Lom\infty]{z(\cdot,t)}\leq c_8 \left(\norm[L^\infty((0,T);\Lom p)]{f}+\norm[L^\infty((0,T);\Lom q)]{g}+ \norm[\Lom\infty]{z_0}+\norm[L^\infty((0,T);\Lom q)]{z}\right).
 \]
\end{lem}
\begin{proof}
 Aided by $L^p$-$L^q$ estimates for the heat semigroup similar to those in \cite[Lemma 1.3 i) and iv)]{win_aggregationvs}, we let $c_9>0$ and $c_{10}>0$ be such that 
 \begin{align*}
  \norm[\Lom \infty]{e^{t\Delta}w}&\leq c_9 (1+t^{-\frac n{2q}})\norm[\Lom q]{w} \qquad \text{for all } w\in \Lom q\\
 \norm[\Lom \infty]{e^{t\Delta}\na\cdot \varphi}&\leq c_{10}(1+t^{-\frac12-\frac n{2p}})e^{-\lambda_1 t} \norm[\Lom{p}]{\varphi}\qquad \text{ for all } \varphi\in L^p(\Om;\R^n).
 \end{align*}
 Then for $t\in(0,2]\cap(0,T)$ we obtain 
 \begin{align*}
  \norm[\Lom\infty]{z(\cdot,t)}&\leq \norm[\Lom \infty]{z_0}+c_{10}\int_0^t (1+(t-s)^{-\frac12-\frac n{2p}})e^{-\lambda_1(t-s)}\norm[\Lom{p}]{f(\cdot,s)}ds \\
&\qquad + c_9\int_0^t(1+(t-s)^{-\frac n{2q}})\norm[\Lom q]{g(\cdot,s)}ds\\
   &\leq \left(1+c_{10}\int_0^2(1+\tau^{-\frac12-\frac n{2p}})e^{-\lambda_1\tau} d\tau + c_9\int_0^2(1+\tau^{-\frac{n}{2q}})d\tau\right)\\
&\qquad\cdot\left(\norm[\Lom\infty]{z_0}+\norm[L^\infty((0,T);\Lom{p})]{f}+\norm[L^\infty((0,T);\Lom q)]{g}\right).
 \end{align*}
 For $t\in(2,\infty)\cap(0,T)$, on the other hand, we have 
 \[
  z(\cdot,t)=e^{\Delta}z(\cdot,t-1)+\int_0^1e^{(1-s)\Delta}\na\cdot f(\cdot,t-1+s) ds + \int_0^1 e^{(1-s)\Delta} g(\cdot,t-1+s)ds
 \]
 and hence may estimate
 \begin{align*}
  \norm[\Lom\infty]{z(\cdot,t)}&\leq 2c_9\norm[\Lom q]{z(\cdot,t-1)} +\int_0^1 c_9(1+(1-s)^{-\frac12-\frac n{2p}})e^{-\lambda_1(1-s)}\norm[\Lom{p}]{f(\cdot,t-1+s)} ds\\&\qquad + \int_0^1 c_9(1+(1-s)^{-\frac n{2q}})\norm[\Lom q]{g(\cdot,t-1+s)} ds\\
  &\leq \left(2c_9+\int_0^1 c_{10}(1+\tau^{-\frac12-\frac n{2p}})e^{-\lambda_1 \tau}d\tau +\int_0^1 c_9(1+\tau^{-\frac n{2q}})d\tau\right)\\
  &\qquad\cdot\left(\norm[L^\infty((0,T);\Lom q)]{z}+\norm[L^\infty((0,T);L^{p}(\Om))]{f}+\norm[L^\infty((0,T);\Lom q)]{g}\right).\qedhere
 \end{align*}
\end{proof}

A consequence for the model under consideration is the following.

\begin{cor}\label{cor:lqtolinfty.fordifference} 
Let $k>\frac n4$ be an integer. 
 There is $c_{11}>0$ such that for any $\eps\in [0,\frac{2\mu}n)$ and any $t>0$ 
 \[
  \norm[\Lom\infty]{\ue(\cdot,t)-u(\cdot,t)}\leq c_{11} (\eps + \norm[L^\infty((0,t);\Lom {2k})]{\ue-u}).
 \]
\end{cor}
\begin{proof}
 The function $\omeps:=\ue-u$ solves $({\omeps})_t=\Delta \omeps + \eps \na \cdot (\ue \na \ve) + \mu \omeps-\mu\omeps(\ue+u)$ and hence from Lemma \ref{lem:lqtolinfty} we can take $c_{12}>0$ such that 
 \begin{align*}
  \norm[\Lom\infty]{\omeps(\cdot,\tau)}&\leq c_{12}(\eps\norm[\Lom\infty]{\ue\na\ve}
  +\norm[L^\infty((0,t);\Lom {2k})]{\mu \omeps-\mu\omeps(\ue+u)}+\norm[L^\infty((0,t);\Lom {2k})]{\omeps})\\
  &\leq c_{12}\eps \norm[L^\infty(\Om\times(0,t))]{\ue\na \ve}
  \\&\qquad + c_{12} (\mu(1+\norm[\Lom\infty]{u}+\norm[\Lom\infty]{\ue})+1)\norm[L^\infty((0,t);\Lom {2k})]{\omeps}
 \end{align*}
 holds for any $\eps\in [0,\frac{2\mu}n)$ and any $\tau\in(0,t)$. Using the uniform bounds on $u$, $\ue$, $\na \ve$ that have been provided by Lemma \ref{lemu} and Lemma \ref{lemnablav}, we obtain the conclusion. 
\end{proof}

\begin{cor}\label{cor:unifconv}
 For any $T>0$ 
 \[
  \ue\to u\qquad \text{uniformly in } \Ombar\times[0,T] \text{ as } \eps\searrow 0.
 \]
\end{cor}
\begin{proof}
 This results from straightforward combination of Corollay \ref{cor:norm2k} and Corollary \ref{cor:lqtolinfty.fordifference}.
\end{proof}

%
%
\section{Large time behaviour in both systems}\label{sec:largetime}
Corollary \ref{cor:unifconv} takes care of convergence of $\ue$ to $u$ on finite time intervals. Seeing that $\ue$ and $u$ both converge to $1$ as $t\to\infty$, we still have hope that they will be close to each other on intervals of the form $(T,\infty)$. We will, nevertheless, need such information in a much more quantitative form -- and this is what we prepare in the present section. After recalling a well-known estimate for the Laplacian in $\Om$ supplemented with homogeneous Neumann boundary conditions, we obtain bounds for $U_\eps$ in the domain of some fractional power of this operator and of $\Delta \ve$ in $\Lom\infty$ that, together with the uniform lower bound of $\ue(\cdot,t)$ for some positive time $t$ (Lemma \ref{lowforeps}), can consequently be turned into precisely those quantitative lower bounds for $\ue$ and $u$ we will need in the proof of Theorem \ref{maintheorem} in Section \ref{sec:proofofthm1}. \\[5pt]

We fix any number $\muhat\in(0,\mu)\cap(0,1)$ and given $p>1$ we let $A=A_p$ denote the realization of the operator 
$-\Delta+\muhat$ under homogeneous Neumann boundary condition in $L^p(\Omega)$.
\begin{lem}
$A$ is sectorial and thus 
possesses closed fractional powers $A^\eta$ for 
arbitrary $\eta>0$, and the corresponding domains 
$D(A^\eta)$ are known to have the embedding property
\be{semi2}
D(A^\eta)\hookrightarrow W^{2,\infty}(\Omega)
\quad if\ 2\eta-\frac{n}{p}>2.
\ee
Moreover, if $(e^{-tA})_{t\geq 0}$ denotes the corresponding 
analytic semigroup, then for each $\eta>0$ 
there exists $c_{13}(p,\eta)>0$ such that 
\be{semi2.5}
  \lp{p}{A^{\eta}e^{-tA}\varphi}
  \leq c_{13}(p,\eta)t^{-\eta}\lp{p}{\varphi}
\ee
for all $t>0$ and each $\varphi\in L^p(\Omega)$. 
\end{lem}
\begin{proof}
 See \cite[Theorem 1.4.3 and Theorem 1.6.1]{henry}.
\end{proof}

%
\begin{lem}\label{lemAu}
For all $p>1$ and any $\eta\in(0,\frac{1}{2})$ 
there exists $c_{14}>0$ such that 
\be{Auineq}
  \lp{p}{A^\eta U_\eps(\cdot,t)}\leq c_{14}
\ee
holds for all $t\geq 2$ and 
all $\eps\in [0,\frac{2\mu}{n})$, 
where $U_\eps\equiv u_\eps-1$. 
\end{lem}
\begin{proof}
According to standard estimates 
for the Neumann heat semigroup (see e.g. \cite[Lemma 3.3]{fiwy}), 
we can find $c_{15}>0$ such that 
\be{semi3}
  \lp{p}{e^{\tau \Delta}\nabla\cdot \varphi}
  \leq c_{15}(1+\tau^{-\frac{1}{2}})\lp{p}{\varphi}
\ee
all $\tau>0$ and any 
$\varphi \in C^1(\overline{\Omega};\R^n)$ 
such that $\varphi\cdot \nu=0$ on $\partial \Omega$. 
We represent $U_\eps$ according to 
\bas
  U_\eps(\cdot,t)
  &=&e^{(t-1)(\Delta-\mu)}U_\eps(\cdot,1)
  -\eps\int^t_1 
  e^{(t-s)(\Delta-\mu)}
  \nabla\cdot (u_\eps\nabla v_\eps(\cdot,s))ds
\\&& -\mu \int^t_1e^{(t-s)(\Delta-\mu)}U_\eps^2(\cdot,s)ds\\
  &=&e^{-(\mu-\muhat)(t-1)}e^{-(t-1)A}U_\eps(\cdot,1)
  -\eps\int^t_{1} 
  e^{-(\mu-\frac{\muhat}{2})(t-s)}
  e^{-\frac{t-s}{2}A}
  e^{\frac{t-s}{2}\Delta}
  \nabla\cdot (u_\eps\nabla v_\eps(\cdot,s))ds
\\
  &&-\mu \int^t_{1}
  e^{-(\mu-\muhat)(t-s)}e^{-(t-s)A}U_\eps^2(\cdot,s)ds
  \qquad \mbox{\rm for all }t>1.
\eas
Hence we can calculate that 
\bas
  \lp{p}{A^\eta U_\eps(\cdot,t)}
  &\leq& e^{-(\mu-\muhat)(t-1)}
  \lp{p}{A^\eta e^{-(t-1)A}U_\eps(\cdot,1)}
\\
  &&+\eps\int^t_{1} 
  e^{-(\mu-\frac{\muhat}{2})(t-s)}
  \lp{p}{A^\eta e^{-\frac{t-s}{2}A}
  e^{\frac{t-s}{2}\Delta}
  \nabla\cdot (u_\eps\nabla v_\eps(\cdot,s))}ds
\\
  &&+\mu \int^t_{1}
  e^{-(\mu-\muhat)(t-s)}\lp{p}{A^\eta e^{-(t-s)A}U_\eps^2(\cdot,s)}ds
  \qquad \mbox{\rm for all } t>1.
\eas
Herein, \eqref{semi2.5} and \eqref{semi3} allow us to estimate 
\begin{align*}
  \norm[\Lom p]{A^\eta e^{-(t-1)A} U_\eps(\cdot,1)}&\leq c_{13}(p,\eta) (t-1)^{-\eta}\norm[\Lom p]{U_\eps(\cdot,1)}, \\
  \norm[\Lom p]{A^\eta e^{-\frac{t-s}2 A } e^{\frac{t-s}2\Delta} \nabla \cdot(\ue\na & \ve (\cdot,s))}
   \\&\leq c_{13}(p,\eta)c_{15}\left(\frac{t-s}2\right)^{-\eta} \left(1+\left(\frac{t-s}2\right)^{-\frac12}\right)\norm[\Lom p]{\ue\na\ve(\cdot,s)},\\
 \norm[\Lom p]{A^\eta e^{-(t-s)A}U_\eps^2(\cdot,s)} &\leq c_{13}(p,\eta) (t-s)^{-\eta}\norm[\Lom p]{U_\eps^2(\cdot,s)} ds
\end{align*}
for $1<s<t$. 
Together with the finiteness of $\int_0^\infty e^{-(\mu-\frac{\muhat}2)\sigma} \sigma^{-\eta} (1+\sigma^{-\frac12}) d\sigma$ and $\int_0^\infty e^{-(\mu-\muhat)\sigma}\sigma^{-\eta}d\sigma$, 
Lemma \ref{lemu} and Lemma \ref{lemnablav} 
establish the existence of $c_{14}>0$ such that \eqref{Auineq} holds for all $t\geq 2$.
%
\end{proof}
%
%
%
%
An important consequence of this estimate is that it provides some control over $\Delta \ve$:
\begin{lem}\label{lemdeltav} 
There exists $c_{16}>0$ such that for all $\eps\in[0,\frac{2\mu}n)$ 
\be{deltavineq} 
  \lp{\infty}{\Delta v_\eps(\cdot,t)}
  \leq c_{16}
\ee
for all $t\geq 3$.
\end{lem}
\begin{proof}
We fix an arbitrary $\gamma\in(1,\frac{3}{2})$ and 
then can choose positive numbers $\eta$ and $p$ 
such that 
\bas
\gamma-1<\eta<\frac{1}{2},
\quad
p>\frac{n}{2(\gamma-1)}. 
\eas
Then $2\gamma-\frac{n}{p}>2\gamma-2(\gamma-1)=2$.
According to a variation-of-constants formula 
associated with the second equation in (\ref{cp3}), 
we can write
\bas
  V_\eps(t)&=&e^{(t-2)(\Delta-1)}V_\eps(\cdot,2)
  +\int^t_2e^{(t-s)(\Delta-1)}U_\eps(\cdot,s)\,ds
\\
  &=&e^{-(1-\muhat)(t-2)}e^{-(t-2)A}V_\eps(2)
  +\int^t_2e^{-(1-\muhat)(t-s)}e^{-(t-s)A}U_\eps(\cdot,s)\,ds, 
\eas
for all $t\geq 2$, and 
hence (\ref{semi2}) implies that 
\bas
  \|V_\eps(\cdot,t)\|_{W^{2,\infty}(\Omega)}
  &\leq& c_{17}\lp{p}{A^\gamma V_\eps(\cdot,t)}
\\
  &\leq&
  c_{17}e^{-(1-\muhat)(t-2)}
  \lp{p}{A^\gamma e^{-(t-2)A}V_\eps(\cdot,2)}
\\
  &&+c_{17}\int^t_2 e^{-(1-\muhat)(t-s)}
  \lp{p}{A^\gamma e^{-(t-s)A}U_\eps(\cdot,s)}\,ds
\\
  &=& c_{17}e^{-(1-\muhat)(t-2)}
  \lp{p}{A^\gamma e^{-(t-2)A}V_\eps(\cdot,2)}
\\
  &&+c_{17}\int^t_2 e^{-(1-\muhat)(t-s)}
  \lp{p}{A^{\gamma-\eta} e^{-(t-s)A}A^\eta U_\eps(\cdot,s)}\,ds
\eas
with some $c_{17}>0$. 
Using (\ref{Auineq}) and (\ref{semi2.5}) to estimate \[\norm[\Lom p]{A^{\gamma-\eta}e^{-(t-s)A}A^\eta U_\eps(\cdot,s)}\leq c_{13}(p,\gamma-\eta) (t-s)^{-(\gamma-\eta)} c_{14}\quad \text{ for any } 2<s<t,\]
and taking into account the boundedness of $c_{17}e^{-(1-\muhat)(\cdot -2)}
  \lp{p}{A^\gamma e^{-(\cdot -2)A}V_\eps(\cdot,2)}$ on $(3,\infty)$ due to 
 \begin{align*}
   &\lp{p}{A^\gamma e^{-(t-2)A}V_\eps(\cdot,2)}\leq c_{13}(p,\gamma)(t-2)^{-\gamma}\norm[\Lom p]{V_\eps(\cdot,2)}\\&\quad\leq c_{13}(p,\gamma)\left(\norm[\Lom p]{e^{2(\Delta-1)}(\vinit-1)} + \int_0^2 \norm[\Lom p]{e^{(2-s)(\Delta-1)} U_\eps(\cdot, s)} ds\right), t\in(3,\infty),
 \end{align*}
and Lemma \ref{lemu},
we obtain $c_{16}>0$ such that \eqref{deltavineq} holds. 
\end{proof}
We will now establish lower bounds for $u_\eps(x,3)$, 
which are independent of $\eps$. In contrast to the final assertion of Theorem \ref{maintheorem}, this lemma relies on our assumption $\uinit\not\equiv0$.
\begin{lem}\label{lowforeps}
There exist 
$c_{18}>0$ and $\eps_1\in(0,\frac{2\mu}{n})$ such that 
\bas
u_\eps(x,3)\geq c_{18}
\eas
for all $x\in\Omega$ and for all $\eps\in[0,\eps_1)$.
\end{lem}
\begin{proof}
We will use a contradiction argument. 
If there exist $(\eps_j)_{j\in\N}$ 
with $\lim_{j\to\infty}\eps_j=0$ and 
$(x_j)_j\subset\Omega$ such that 
\be{epsx}
\lim_{j\to\infty}u_{\eps_j}(x_j,3)=0, 
\ee
then we can take a subsequence $(x_{j_k})_k\subset(x_{j})_j$ 
satisfying that 
there exists $x_0\in\overline{\Omega}$ such that 
\bas
  \lim_{k\to\infty}x_{j_k}=x_0. 
\eas
Thanks to Corollary \ref{cor:unifconv} 
 and (\ref{epsx}) we deduce 
\bas
  u(x_0,3)
  =\lim_{k\to\infty}u_{\eps_{j_k}}(x_{j_k},3)
  =0.
\eas
However, we obtain 
by the strong maximum principle that
\bas
u(x_0,3)>0,
\eas
which is contradiction. 
\end{proof}
We are now able to estimate $\ue$ and $u$ from below. The following lemma can be viewed as a one-sided quantitative statement on the long-term behaviour of $\ue$ and $u$.
\begin{lem}\label{lemlowu}
Let $\eps_{2}:=\min\{\eps_1,\frac{\mu}{2c_{16}}\}\leq\frac{2\mu}{n}$ with $\eps_1$ taken from Lemma \ref{lowforeps} and $c_{16}$ as defined in Lemma \ref{lemdeltav}. 
Then there is $c_{19}>0$ such that 
\begin{equation}\label{eq:infinf}
 \inf_{x\in\Om} \inf_{\eps\in[0,\eps_0)} \ue(x,t)\geq \frac{1-\frac{c_{16}\eps}{\mu}}{1+c_{19}e^{-\frac{\mu}2 t}}\quad \text{for all } t\geq 3.
\end{equation}
\end{lem}
\begin{proof}
From the first equation in (\ref{cp}) 
and (\ref{deltavineq}) we see that
\bas
  (u_{\eps})_t&=&
  \Delta u_\eps
  -\eps \nabla u_\eps\cdot\nabla v_\eps
  -\eps u_\eps\Delta v_\eps+\mu u_\eps-\mu u_\eps^2
\\
  &\geq& 
  \Delta u_\eps
  -\eps \nabla u_\eps\cdot\nabla v_\eps
  +(\mu-c_{16}\eps) u_\eps-\mu u_\eps^2 \quad \text{ in } \Om
\eas
for $t\geq 3$. We choose $y_0\in(0,\frac12)$ to  
be a positive number 
such that $y_0<\frac{\mu-\eps_0c_{16}}{\mu}$ and  $y_0\leq\inf_{\eps\in[0,\eps_0)}\inf_{x\in\Omega}u_\eps(x,3)$ 
(cf. Lemma \ref{lowforeps}) and put 
\bas
\ye (t):=\frac{\mu-c_{16}\eps}{\mu+(\frac{\mu-c_{16}\eps}{y_0}-\mu)e^{-(\mu-c_{16}\eps)(t-3)}}, \quad t\geq 3.
\eas
Then $\ye :[3,\infty)\to\R$ is the solution to 
the ODE initial value problem
\bas
\left\{
\begin{array}{l}
  \ye '(t)=(\mu-c_{16}\eps)\ye (t)-\mu \ye (t)^2,
  \quad t>3, \\
  \ye(3)=y_0.
\end{array}
\right.
\eas

Apparently, $0<\frac{\mu-c_{16}\eps}{y_0}-\mu\leq \frac{\mu}{y_0}$, and due to $\eps<\frac{\mu}{2c_{16}}$, we may employ the estimate $e^{-(\mu-c_{16}\eps)(t-3)}\leq e^{-\frac{\mu}2t+\frac{3\mu}2}$ to see that 
\[
 \ye(t)\geq\frac{\mu-c_{16}\eps}{\mu+\frac{\mu}{y_0}e^{\frac{3\mu}2}e^{-\frac{\mu}2t}}=\frac{1-\frac{c_{16}\eps}{\mu}}{1+c_{19} e^{-\frac{\mu}2t}}\quad\text{for all } t\geq 3
\]
if we let $c_{19}:=\frac1{y_0}e^{\frac{3\mu}2}$. In light of a comparison lemma we can deduce \eqref{eq:infinf}.
\end{proof}

\section{Global-in-time convergence: Proof of Theorem \ref{maintheorem}} \label{sec:proofofthm1}
With these explicit and quantitative uniform lower bounds for $\ue$ and $u$, everything has been prepared to revisit the differential inequality of Lemma \ref{lem:diffineq} and turn our attention to the proof of Theorem \ref{maintheorem}. 
In fact, we only have to show the following:
\begin{lem}\label{lem:linearconvergence}
 There are $c_{20}>0$ and $\eps_0>0$ such that 
 \[
  \norm[\Lom{\infty}]{\ue(\cdot,t)-u(\cdot,t)}\leq c_{20}\eps \qquad \text{for any } t\in(0,\infty) \text{ and any } \eps\in[0,\eps_0).
 \]
\end{lem}
\begin{proof}
We let $k>\frac n4$ be an integer and with $\eps_2$ as in Lemma \ref{lemlowu} and $c_{16}$ taken from Lemma \ref{lemdeltav} we set  $\eps_0:=\min\{\eps_2,\frac{\mu}{8c_{16}}\}$.
 In accordance with Lemma \ref{lemlowu}, we then choose $T>0$ such that on $\Om\times(T,\infty)$ 
 \[
  u>\frac78 \quad \text{and}\quad \ue>\frac78\quad \text{for any } \eps\in[0,\eps_0).
 \]
 By Corollary \ref{cor:norm2k} and Corollary \ref{cor:lqtolinfty.fordifference}  
there is $c_{21}>0$ such that 
 \begin{equation}\label{eq:estimate.smallt}
  \norm[L^\infty(\Om\times(0,T))]{\ue-u}\leq c_{21}\eps\qquad \text{for all } \eps\in[0,\eps_0).
 \end{equation}
 Moreover, Lemma \ref{lem:diffineq} ensures that (with $c_6(k)$ as defined there) 
 \begin{align*}
  \frac{d}{dt} \io \omeps^{2k} &\leq \eps^{2k} c_6(k) + \mu k\io \omeps^{2k}+2\mu k\io (1-u-\ue)\omeps^{2k}\\
  &\leq \eps^{2k} c_6(k) - \frac{\mu k}2 \io \omeps^{2k}\quad \text{for all } \eps\in[0,\eps_0), t\in(T,\infty),
 \end{align*}
 where we have used that $1-u-\ue<1-\frac78-\frac78=-\frac34$ on $\Om\times(T,\infty)$. Therefore, 
 \[
  \io\omeps^{2k}(\cdot,t)\leq \io \omeps^{2k} (\cdot,T)e^{-\frac{\mu k}2(t-T)}+\frac{2c_6(k)}{\mu k} \eps^{2k}\leq \left(c_{21}^{2k}+\frac{2c_6(k)}{\mu k}\right)\eps^{2k}
 \]
 for all $t>T$ and all $\eps\in[0,\eps_0)$, and hence we have found $c_{22}>0$ such that 
 \[
  \norm[\Lom {2k}]{\ue(\cdot,t)-u(\cdot,t)}\leq c_{22}\eps\quad \text{for all } t>T, \eps\in[0,\eps_0).
 \]
 A further application of Corollary \ref{cor:lqtolinfty.fordifference}
shows that hence for some $c_{23}>0$ 
 \[
  \norm[\Lom \infty]{\ue(\cdot,t)-u(\cdot,t)}\leq c_{23}\eps \qquad \text{for all } t>T, \eps\in[0,\eps_0).
 \]
Together with \eqref{eq:estimate.smallt}, this proves the lemma and hence also Theorem \ref{maintheorem}. 
\end{proof}

\section*{Acknowledgements}
 J.~Lankeit acknowledges support of the {\em Deutsche Forschungsgemeinschaft} within the project {\em Analysis of chemotactic cross-diffusion in complex
 frameworks}. 

{\footnotesize

\begin{thebibliography}{10}

\bibitem{ai}
S.~Ai.
\newblock Traveling wave fronts for generalized {F}isher equations with
  spatio-temporal delays.
\newblock {\em J. Differential Equations}, 232(1):104--133, 2007.

\bibitem{aronson_weinberger}
D.~G. Aronson and H.~F. Weinberger.
\newblock Multidimensional nonlinear diffusion arising in population genetics.
\newblock {\em Adv. in Math.}, 30(1):33--76, 1978.

\bibitem{BBTW}
N.~Bellomo, A.~Bellouquid, Y.~Tao, and M.~Winkler.
\newblock Toward a mathematical theory of {K}eller-{S}egel models of pattern
  formation in biological tissues.
\newblock {\em Math. Models Methods Appl. Sci.}, 25(9):1663--1763, 2015.

\bibitem{berestycki_hamel_nadirashvili}
H.~Berestycki, F.~Hamel, and N.~Nadirashvili.
\newblock The speed of propagation for {KPP} type problems. {I}. {P}eriodic
  framework.
\newblock {\em J. Eur. Math. Soc. (JEMS)}, 7(2):173--213, 2005.

\bibitem{du_guo}
Y.~Du and Z.~Guo.
\newblock The {S}tefan problem for the {F}isher-{KPP} equation.
\newblock {\em J. Differential Equations}, 253(3):996--1035, 2012.

\bibitem{ei_izuhara_mimura}
S.-I. Ei, H.~Izuhara, and M.~Mimura.
\newblock Spatio-temporal oscillations in the {K}eller-{S}egel system with
  logistic growth.
\newblock {\em Phys. D}, 277:1--21, 2014.

\bibitem{fisher}
R.~A. Fisher.
\newblock The wave of advance of advantageous genes.
\newblock {\em Annals of eugenics}, 7(4):355--369, 1937.

\bibitem{fiwy}
K.~Fujie, A.~Ito, M.~Winkler, and T.~Yokota.
\newblock Stabilization in a chemotaxis model for tumor invasion.
\newblock {\em Discrete Contin. Dyn. Syst.}, 36(1):151--169, 2016.

\bibitem{hamel_ryzhik}
F.~Hamel and L.~Ryzhik.
\newblock On the nonlocal {F}isher-{KPP} equation: steady states, spreading
  speed and global bounds.
\newblock {\em Nonlinearity}, 27(11):2735--2753, 2014.

\bibitem{he_zheng}
X.~He and S.~Zheng.
\newblock Convergence rate estimates of solutions in a higher dimensional
  chemotaxis system with logistic source.
\newblock {\em J. Math. Anal. Appl.}, 436(2):970--982, 2016.

\bibitem{henry}
D.~Henry.
\newblock {\em Geometric theory of semilinear parabolic equations}, volume 840
  of {\em Lecture Notes in Mathematics}.
\newblock Springer-Verlag, Berlin-New York, 1981.

\bibitem{hillenpainter_survey}
T.~Hillen and K.~J. Painter.
\newblock A user's guide to {PDE} models for chemotaxis.
\newblock {\em J. Math. Biol.}, 58(1-2):183--217, 2009.

\bibitem{horstmannI}
D.~Horstmann.
\newblock From 1970 until present: the {K}eller-{S}egel model in chemotaxis and
  its consequences. {I}.
\newblock {\em Jahresber. Deutsch. Math.-Verein.}, 105(3):103--165, 2003.

\bibitem{huang}
R.~Huang.
\newblock Stability of travelling fronts of the {F}isher-{KPP} equation in
  {$\mathbb{R}^N$}.
\newblock {\em NoDEA Nonlinear Differential Equations Appl.}, 15(4-5):599--622,
  2008.

\bibitem{KS}
E.~F. {K}eller and L.~A. {S}egel.
\newblock Initiation of slime mold aggregation viewed as an instability.
\newblock {\em J. Theoret. Biol.}, 26(3):399 -- 415, 1970.

\bibitem{KPP}
A.~Kolmogorov, I.~Petrovskii, and N.~Piskunov.
\newblock A study of the diffusion equation with increase in the amount of
  substance, and its application to a biological problem.
\newblock In V.~M. Tikhomirov, editor, {\em Selected Works of A. N. Kolmogorov:
  Volume I: Mathematics and Mechanics}, pages 242--270. Springer Netherlands,
  Dordrecht, 1991.

\bibitem{lankeit_thresholds}
J.~Lankeit.
\newblock Chemotaxis can prevent thresholds on population density.
\newblock {\em Discrete Contin. Dyn. Syst. Ser. B}, 20(5):1499--1527, 2015.

\bibitem{lankeit_evsmooth}
J.~Lankeit.
\newblock Eventual smoothness and asymptotics in a three-dimensional chemotaxis
  system with logistic source.
\newblock {\em J. Differential Equations}, 258(4):1158--1191, 2015.

\bibitem{lau}
K.-S. Lau.
\newblock On the nonlinear diffusion equation of {K}olmogorov, {P}etrovsky, and
  {P}iscounov.
\newblock {\em J. Differential Equations}, 59(1):44--70, 1985.

\bibitem{lions}
P.~Lions.
\newblock R\'esolution de probl\`emes elliptiques quasilin\'eaires.
\newblock {\em Arch. Rational Mech. Anal.}, 74(4):335--353, 1980.

\bibitem{mimura_tsujikawa}
M.~Mimura and T.~Tsujikawa.
\newblock Aggregating pattern dynamics in a chemotaxis model including growth.
\newblock {\em Physica A: Statistical Mechanics and its Applications},
  230(3):499--543, 1996.

\bibitem{nakaguchi_osaki_2d}
E.~Nakaguchi and K.~Osaki.
\newblock Global existence of solutions to a parabolic-parabolic system for
  chemotaxis with weak degradation.
\newblock {\em Nonlinear Anal.}, 74(1):286--297, 2011.

\bibitem{nakaguchi_osaki_2d3d}
E.~Nakaguchi and K.~Osaki.
\newblock Global solutions and exponential attractors of a parabolic-parabolic
  system for chemotaxis with subquadratic degradation.
\newblock {\em Discrete Contin. Dyn. Syst. Ser. B}, 18(10):2627--2646, 2013.

\bibitem{OTYM}
K.~Osaki, T.~Tsujikawa, A.~Yagi, and M.~Mimura.
\newblock Exponential attractor for a chemotaxis-growth system of equations.
\newblock {\em Nonlinear Anal.}, 51(1, Ser. A: Theory Methods):119--144, 2002.

\bibitem{painterhillen_physD}
K.~J. Painter and T.~Hillen.
\newblock Spatio-temporal chaos in a chemotaxis model.
\newblock {\em Physica D: Nonlinear Phenomena}, 240(4-5):363--375, 2011.

\bibitem{stan_vazquez}
D.~Stan and J.~L. V{\'a}zquez.
\newblock The {F}isher-{KPP} equation with nonlinear fractional diffusion.
\newblock {\em SIAM J. Math. Anal.}, 46(5):3241--3276, 2014.

\bibitem{tello_winkler_07}
J.~I. Tello and M.~Winkler.
\newblock A chemotaxis system with logistic source.
\newblock {\em Comm. Partial Differential Equations}, 32(4-6):849--877, 2007.

\bibitem{Viglialoro}
G.~Viglialoro.
\newblock Very weak global solutions to a parabolic-parabolic chemotaxis-system
  with logistic source.
\newblock {\em J. Math. Anal. Appl.}, 439(1):197--212, 2016.

\bibitem{win_emergence}
M.~Winkler.
\newblock Emergence of large population densities despite logistic growth
  restrictions in fully parabolic chemotaxis systems.
\newblock {\em Discrete Contin. Dyn. Syst. Ser. B}.
\newblock to appear.

\bibitem{win_veryweak}
M.~Winkler.
\newblock Chemotaxis with logistic source: very weak global solutions and their
  boundedness properties.
\newblock {\em J. Math. Anal. Appl.}, 348(2):708--729, 2008.

\bibitem{win_aggregationvs}
M.~Winkler.
\newblock Aggregation vs. global diffusive behavior in the higher-dimensional
  {K}eller-{S}egel model.
\newblock {\em J. Differential Equations}, 248(12):2889--2905, 2010.

\bibitem{W2010}
M.~Winkler.
\newblock Boundedness in the higher-dimensional parabolic-parabolic chemotaxis
  system with logistic source.
\newblock {\em Comm. Partial Differential Equations}, 35(8):1516--1537, 2010.

\bibitem{W2014}
M.~Winkler.
\newblock Global asymptotic stability of constant equilibria in a fully
  parabolic chemotaxis system with strong logistic dampening.
\newblock {\em J. Differential Equations}, 257(4):1056--1077, 2014.

\bibitem{win_exceed}
M.~Winkler.
\newblock How far can chemotactic cross-diffusion enforce exceeding carrying
  capacities?
\newblock {\em J. Nonlinear Sci.}, 24(5):809--855, 2014.

\end{thebibliography}

\end{document}